\documentclass[12pt]{article}
\usepackage{amsmath, amsfonts, amsthm, amssymb, color, hyperref, extarrows, enumitem, nicefrac,blindtext, mathtools,}

\newtheorem{theorem}{Theorem}[section]

\newtheorem{pro}[theorem]{Proposition}

\numberwithin{equation}{section}

\newtheorem{remark}[theorem]{Remark}

\newtheorem{example}{Example}

\usepackage[hyperpageref]{backref}

\usepackage{cite}

\hypersetup{
	colorlinks   = true,
	citecolor    = magenta}
	
\begin{document}
\title{\vspace{-1cm} \bf Sobolev regularity of the canonical solutions to $\bar\partial$ on product domains  \rm}
\author{ Yuan Zhang}
\date{}

\maketitle

\begin{abstract}
 Let $\Omega$ be a product domain in $\mathbb C^n, n\ge 2$, where each slice has smooth boundary. We observe  that the   canonical solution operator for the $\bar\partial$ equation on $\Omega$ is bounded in  $W^{k,p}(\Omega)$, $k\in \mathbb Z^+, 1<p<\infty$.  This Sobolev regularity is sharp in view of Kerzman-type examples.

\end{abstract}

\renewcommand{\thefootnote}{\fnsymbol{footnote}}
\footnotetext{\hspace*{-7mm}
\begin{tabular}{@{}r@{}p{16.5cm}@{}}
& 2010 Mathematics Subject Classification. Primary  32W05; Secondary 32A25, 32A36\\
& Key words and phrases.
canonical solution,   $\bar{\partial}$ equation, Bergman projection,  product domains, Sobolev regularity
\end{tabular}}

\section{Introduction}
Let $\Omega$ be a bounded pseudoconvex domain in $\mathbb C^n, n\ge 1$. According to H\"ormander's $L^2$ theory, given a $\bar\partial$-closed $(0,1)$ form $f\in L^2(\Omega)$, there exists a unique $L^2$ function that is perpendicular to $ ker(\bar\partial)$ and solves  
\begin{equation*} 
    \bar\partial u =f\quad \text{in} \quad \Omega.
\end{equation*}
 This solution is called the {\it canonical solution} (of the $\bar\partial$ equation). The $L^2$-Sobolev regularity of the canonical solutions has been  investigated through  Kohn's $\bar\partial$-Neumann approach  for domains with nice regularity and  geometry, such as convexity and/or finite type conditions. 
 
  The goal of the paper is to give  the $L^p$-Sobolev estimate of the canonical solutions on product domains. Here a product domain  $\Omega$   in $\mathbb C^n$ is a Cartesian product $D_1\times\cdots \times D_n$ of bounded planar domains $D_j, j=1, \ldots, n$. In particular, $D_j$ needs not  be simply-connected.  Then $\Omega$ is   (weakly) pseudoconvex   with at most Lipschitz boundary. The $L^p$ regularity of  the canonical solutions on product domains was already thoroughly understood through works of  \cite{La, CM, DLT, DPZ, Yu, Li, Zh} and the references therein. In the Sobolev category,    combined efforts in \cite{CS, JY, YZ, PZ2} have given the existence of a bounded solution  operator of $\bar\partial$ sending  $W^{k+n-2,p}(\Omega) $ into $W^{k,p}(\Omega) $, $k\in \mathbb Z^+, 1< p< \infty$.   Here   $W^{k,p}(\Omega) $ is the Sobolev space  consisting of functions whose weak derivatives on $\Omega$ up to order $k$ exist and belong to $L^p(\Omega)$. The main theorem is stated as follows.

 \begin{theorem}\label{Tbn}
Let $\Omega:=D_1\times \cdots  \times D_n \subset \mathbb C^n, n\ge 2$, where each $D_j$ is a  bounded domain in $\mathbb C$ with smooth boundary, $j=1, \ldots, n$. Given a $\bar\partial$-closed $(0,1)$ form $f  \in W^{k,p}(\Omega)$,  $k\in \mathbb Z^+, 1<p< \infty,$ the canonical solution $Tf$  of $\bar\partial u =f$ on $\Omega$ is in $W^{k, p}(\Omega)$. Moreover, there exists a constant $C$ dependent only on $\Omega, k$ and $p$ such that 
  \begin{equation*} 
   \|Tf\|_{W^{k , p}(\Omega)}\le C\|f\|_{W^{k, p}(\Omega)}.
\end{equation*}
\end{theorem}

  The proof of Theorem \ref{Tbn} is essentially an observation  on  a representation formula of the canonical solutions    
introduced by Li \cite{Li}, according to which  it  boils down to the  Sobolev estimates of  the Bergman projection  and   canonical solution operators on planar domains.  On the other hand, with an application of a formula of Spencer  on planar domains, the    Sobolev estimates of these two operators are simply a consequence of a   result of Jerison and Kenig in \cite{JK}. 
In Example \ref{ex}, a  datum $f$  on the bidisc is constructed, such that  $f\in  W^{k,   q} $  for all $1<q<p$, yet   $\bar\partial u =f$    has no $ W^{k, p}$ solutions.  This example indicates that  the $\bar\partial$ problem   does not gain Sobolev regularity on product domains in general, and thus the  estimate in Theorem \ref{Tbn} is sharp.


\medskip
\noindent{\bf Acknowledgement:} The author thanks Professor Song-Ying Li for helpful comments and suggestions.

\section{Bergman projection and canonical solutions on  planar domains}
Let $D$ be a bounded domain in $\mathbb C$ whose boundary $bD$ is smooth, and   $g$ be the  Green's function on $D$. In other words, at a fixed pole  $w \in D$,  
$$
g(z, w): = -\frac{1}{2\pi}\sup \left\{ u(z):   u \in SH^-(D)\ \text{and}\  \limsup_{\zeta \to w} (u(\zeta) - \log |\zeta - w|) < \infty \right\}, \quad z\in D,
$$
where $SH^-(D)$ is the collection of negative subharmonic functions on $D$. It is known (\cite{Ev} etc.) that $g$ is symmetric on the two variables $z$ and $w$. Moreover, there exists a harmonic function $h_w $ on $D$ with $  h_w  =  \frac{1}{2\pi}\ln|\cdot -w|$ on $bD$ such that 
\begin{equation}\label{g}
    g(\cdot, w) = -\frac{1}{2\pi}\ln|\cdot -w| + h_w\quad \text{in} \quad D.
\end{equation}
In particular, $h_w\in C^\infty(D)$ and 
\begin{equation}\label{gb}
    g(z, w) =g(w, z)=0, \quad z\in bD.
\end{equation}

Given $f\in L^p(D), 1<p<\infty$, define
\begin{equation}\label{G}
Gf: =   -4\int_{D}   g(\cdot, w)  f(w) d\nu_w \quad \text{in}\quad D.
\end{equation}
 Here $d\nu$ is the Lebesgue measure on $\mathbb C$. Then $G f$  is the solution to the Dirichlet problem
\begin{equation*} \label{D}
\begin{cases}
 \Delta u  = 4  f , &\quad \text{in} \quad D; \\
u =0, & \quad \text{on}\quad  bD .
\end{cases}
\end{equation*}
Moreover,    
 $G$ is a bounded operator sending $W^{\alpha-2, p}(D)$ into $W^{\alpha, p}(D),  1<p<\infty, \alpha> \frac{1}{p}$. See  \cite[Theorem 0.3]{JK} by Jerison and Kenig. In particular, if  $f\in W^{k-1, p}(D), k\in\mathbb Z^+\cup\{ 0\}, 1<p<\infty$, then
\begin{equation}\label{Gb}
\|Gf\|_{W^{k+1, p}(D)}    \lesssim \|f\|_{W^{k-1, p}(D)}.  
\end{equation}
Here and throughout the rest of the paper, we say two quantities $a$ and $b$ to satisfy $a\lesssim b$ if there exists a constant $C$ dependent only possibly on the underlying domain,  $k$ and $p$ such that $a\le Cb$.
 \medskip
 
The Bergman projection  operator $P$  on a domain $\Omega$ is the orthogonal projection of   $L^2(\Omega)$    onto the Bergman space $A^2(\Omega)$,  the space of $L^2$  holomorphic functions on $\Omega$.  Since  $A^2(\Omega )$ is a reproducing kernel Hilbert space, there exists a function $k: \Omega \times \Omega \rightarrow \mathbb C$, called the Bergman kernel, such that for all $f\in L^2(\Omega)$, 
\begin{equation*} 
 P f = \int_{\Omega } k (\cdot, w) f(w) {d \nu_w}  \quad \text{in}\quad \Omega.
\end{equation*}
On a smooth planar domain $D$, the Bergman kernel $k$ is related to the Green's function $g$ by
\begin{equation}\label{bg}
k(z, w)= {-4}  \partial_z \partial_{ \bar w}g(z, w), \ \ z\ne w\in D.
\end{equation}
 See \cite[pp 180]{Be}. Clearly,   $k(\cdot, w) \in C^\infty(\bar D)$ by \eqref{g}. 
 
 If $D$ is  simply-connected,  the Sobolev boundedness of the Bergman projection $P$ can be obtained   by   applying the known Sobolev  regularity  on the unit disc and the Riemann mapping theorem. On general smooth planar domains, Lanzani and Stein    suggested an approach to   estimate  $P$ briefly in \cite{LS}. For completeness and convenience of the reader, the detail of their approach  to the Sobolev regularity of $P$ is provided below.


\begin{theorem}\label{Pb}
Let $D\subset \mathbb C$ be a  bounded domain with $C^\infty$ boundary.
Then  the Bergman projection $P$ is (or, extends as) a bounded operator on $W^{k,p}(D)$,  $k\in \mathbb Z^+\cup\{0\}$,   $ 1< p<\infty$. Namely,  for any $f\in W^{k, p}(D)$,
\begin{equation*} 
    \|Pf\|_{W^{k, p}(D)}\lesssim \|f\|_{W^{k, p}(D)}.
\end{equation*}


\end{theorem}

\begin{proof}
We shall need  the following Spencer's formula:  for any $f\in L^2( D)$, 
\begin{equation}\label{ps}
     Pf   +   \partial G \bar\partial f = f  \quad \text{in}\quad  D,
\end{equation}
where $G$ is defined in \eqref{G}. Note that $\partial G \bar\partial f $ is well-defined by \eqref{Gb}. The proof of \eqref{ps} can be found, for instance, in \cite[pp. 73-75]{Be}. Here we give  a short and direct proof. First  assume   $f\in C^\infty(\bar D)$.  Fix $z\in D$ and let $ D(z;\epsilon)$ be the disc centered at $z$ with radius $\epsilon>0$.   By  \eqref{bg} and the Stokes' theorem, 
\begin{equation*}
    \begin{split}
        Pf(z)  + \partial G \bar \partial f(z) = &-4\lim_{\epsilon\rightarrow 0}\left(\int_{D\setminus D(z;\epsilon)}  \partial_{\bar w}   {\partial_z}g(z, w) f(w) d\nu_w - \int_{D\setminus D(z;\epsilon)}   {\partial_z}g(z, w)  \partial_{\bar w}f(w) d\nu_w\right)\\
        =& 2i\int_{ bD }  \partial_z g(z, w) f(w) dw -2i \lim_{\epsilon\rightarrow 0}\int_{ bD (z;\epsilon)}    {\partial_z}g(z, w) f(w) dw =: I_1+I_2.
    \end{split}
\end{equation*}
Making use of \eqref{gb} we have
$$I_1 = 2i \partial_z \int_{ bD }  g(z, w) f(w) dw =0.  $$
For $I_2$, note that by \eqref{g} there exists  some bounded function $  h$ on $ D(z;\epsilon) $ such that $$ \partial_z g(z, \cdot) = \frac{1}{4\pi(\cdot-z)} +  h \quad \text{in}\quad D(z;\epsilon).$$
 By continuity of $f$,
$$ I_2 =-2i\lim_{\epsilon\rightarrow 0}\int_{ bD (z;\epsilon)}  \partial_z g(z, w) f(w) dw = \frac{1}{2\pi i} \lim_{\epsilon\rightarrow 0}\int_{ bD (z;\epsilon)}  \frac{ f(w)}{w-z} dw = f(z). $$
Hence \eqref{ps} holds for $f\in C^\infty(\bar D)$. For general $f\in L^2( D)$, choose $\{f_m\}_{m=1}^\infty\in C^\infty(\bar D)$ that converges to  $f$ in $L^2( D)$ norm as $m\rightarrow \infty$ (see, for instance, \cite[pp. 268]{Ev}).  Employing a standard density argument, we obtain \eqref{ps} by  the trivial boundedness of $P$  in $L^2(D)$, and the estimate \eqref{Gb} for $ G $ with $k=0$ and $ p=2$.

For $f\in  L^p(D), 1<p<2$,  one uses \eqref{ps} to extend $P$ on $f$ by  defining  $Pf: =  \lim_{m\rightarrow\infty } f_m-\partial G \bar\partial f_m$, where the family $\{f_m\}_{m=1}^\infty\in L^2(D)$ converges to $f$ in $L^p(D)$ norm. Note that this limit exists and is independent of the choice of $\{f_m\}_{m=1}^\infty$ due to \eqref{Gb}.  Moreover, for all $1<p<\infty$,  by \eqref{Gb} the extended operator $P$ satisfies
\begin{equation*}
    \begin{split}
        \left\|{P} f\right\|_{W^{k, p}(D)}\lesssim &\|f\|_{W^{k, p}(D)} +  \left\|G\bar\partial f \right\|_{W^{k+1, p}(D)} \lesssim  \|f\|_{W^{k, p}(D)} +  \left\| \bar\partial f \right\|_{W^{k-1, p}(D)} \lesssim  \|f\|_{W^{k, p}(D)}.
    \end{split}
\end{equation*}
This  completes the proof of the theorem.

\end{proof}

Given $f\in L^p(D), 1<p<\infty$, define  \begin{equation}\label{T}
     T f: = \partial G f\left( = -4\partial \int_{D}   g(\cdot, w)  f(w) d\nu_w\right)   \quad \text{in}\quad  D. \end{equation}
 Then $T$ is the canonical solution of $\bar\partial $ on $D$. Indeed,  by \eqref{D} one  first has $\bar\partial T   f = \bar\partial\partial Gf =f$ on $D$.  On the other hand,  for any $h\in A^2(D), $ 
$$ \langle T  f, h  \rangle = \langle T \bar\partial T f, h  \rangle=\langle Tf-P Tf, h  \rangle = \langle Tf-PTf, P h  \rangle =  \langle PT f-PT f,  h  \rangle =0,$$
implying $ T  f\perp A^2(\Omega) $. Here  in the first equality we used the fact that  $ \bar\partial T   f =f$ on $D$;  
in the second equality we used \eqref{ps} with $f$ replaced by $Tf$; in the third equality we used the fact that $Ph=h$ when $h\in A^2(D)$; in the fourth equality we used the  projection properties of $P$, i.e., $P^*=P=P^2$.  
The  Sobolev regularity of   $T $ below follows  immediately from   \eqref{Gb} and \eqref{T}.  
\medskip

\begin{theorem}\label{Tb}
Let $D$ be a bounded domain in $\mathbb C$ with smooth boundary. For each $k\in \mathbb Z^+\cup\{0\}, 1<p<\infty $,     the canonical solution operator $T$   of $\bar\partial $   on $D$ defined in \eqref{T}         is  a bounded operator sending $W^{k, p}(D)  $ into $W^{k+1, p}(D) $. Namely,  for any $f\in W^{k, p}(D)$, 
     \begin{equation*} 
   \|Tf\|_{W^{k+1, p}(D)}\lesssim \|f\|_{W^{k, p}(D)}.
\end{equation*}

\end{theorem}
 
\medskip

\begin{remark}
a). We can further make use of Theorem \ref{Tb} and  the Sobolev embedding theorem to conclude that   the canonical solution operator  $T$ sends $W^{k, \infty}(D)$ into $C^{k, \alpha}(D)$ for all $0<\alpha<1$  with 
\begin{equation*} 
 \|Tf\|_{C^{k,\alpha}(D) } \le C \|f\|_{W^{k, \infty}(D) },  
\end{equation*}
where $C$ depends only on $D, k$ and $\alpha$. In particular, this inequality improves a supnorm estimate in  \cite{BL}.\\
b). Another well-known solution operator $\tilde T$ of $\bar\partial $ on $D$ is given in terms of the universal Cauchy kernel as follows. 
$$\tilde Tf: = -\frac{1}{\pi } \int_D \frac{f(w)}{w-\cdot}d\nu_w\quad \text{in}\quad D. $$
It was proved by Prats  in \cite{Pr} that $\tilde T  $ enjoys a similar Sobolev regularity as $T$ (see also \cite{PZ} for a much simpler proof using Caldr\'on-Zygmund's classical singular integral theory):   
$$ \|\tilde Tf\|_{W^{k+1, p}(D)}\lesssim \|f\|_{W^{k, p}(D)}. $$

\end{remark}



\section{Canonical solutions on product domains}

Let $\Omega:=D_1\times \cdots  \times D_n \subset \mathbb C^n$, $n\geq 2$, where each $D_j$ is a bounded planar domain with smooth boundary. Denote by $P_j$   the Bergman projection  operator of $D_j, j=1, \ldots, n$. Then the Bergman projection $P$ of $\Omega$ satisfies
\begin{equation}\label{Pn}
    P = P_1 \cdots   P_n. 
\end{equation} 

Let $T_j$ be the canonical solution operator on $D_j$ defined in \eqref{T}, with $D$ replaced by $D_j$, $j=1, \ldots, n$. Given a $\bar\partial$-closed $(0,1)$ form $f =\sum_{j=1}^nf_jd\bar z_j\in L^p(\Omega)$, it was shown   in \cite[Theorem 2.5]{Li} (or, through a repeated application of \eqref{ps} together with the $\bar\partial$-closedness of $f$) that 
\begin{equation}\label{Tn}
Tf= T_1f_1 +T_2P_1 f_2 +\cdots +T_nP_1\cdots P_{n-1 }f_n    
\end{equation}
is the canonical solution  to $\bar\partial u = f$ on $\Omega$. The following proposition gives the Sobolev boundedness of $T_j$ and $P_j$ on $\Omega$.

\begin{pro}\label{TPj}
Let $\Omega:=D_1\times \cdots  \times D_n \subset \mathbb C^n$, where  each $D_j$ is a  bounded domain in $\mathbb C$ with smooth boundary, $j=1, \ldots n$. Then $T_j$ and $P_j$ are bounded operators in $W^{k, p}(\Omega) $, $ k\in \mathbb Z^+\cup\{0\}, 1<p< \infty$. Namely, for all $f\in W^{k, p}(\Omega)$,
\begin{equation*}
\begin{split}
   & \|T_jf\|_{W^{k , p}(\Omega)}\lesssim \|f\|_{W^{k, p}(\Omega)};     \quad \|P_jf\|_{W^{k , p}(\Omega)}\lesssim \|f\|_{W^{k, p}(\Omega)}.
\end{split}
\end{equation*}
\end{pro}

\begin{proof}
For simplicity yet without loss of generality, assume $j=1$ and $n=2$.   Denote by  $ \nabla_j $   either $\bar\partial_j$ or $\partial_j$  in the $z_j$ variable.  
Since  $\bar\partial_{ 1}T_{1}= id$ and $\bar\partial_1 P_1 =0 $, we only need to prove for all $k_1, k_2\in \mathbb Z^+\cup\{0\}$, $k_1+k_2 =k$,
\begin{equation*}
    \|\partial_1^{k_1}T_1 \nabla_2^{k_2} f\|_{L^{ p}(\Omega)}\lesssim \|f\|_{W^{k, p}(\Omega)}; \quad   \|\partial_1^{k_1} P_1\nabla_2^{k_2}f\|_{L^{p}(\Omega)}\lesssim \|f\|_{W^{k, p}(\Omega)}.
\end{equation*}
In fact, making use of Theorem \ref{Tb} and Fubini Theorem, 
\begin{equation*}
    \begin{split}
        \|\partial_1^{k_1}T_1 \nabla_2^{k_2} f\|^p_{L^{p}(\Omega)}
        = & \int_{D_2} \left\|\partial_{1}^{k_{1}} T_{1}  \left( \nabla_2^{k_2} f\right)(\cdot, w_2) \right\|^p_{L^p(D_{ 1})}d\nu_{w_2}\\
        \lesssim& \sum_{m_{ 1}=0}^{k_{ 1}}\int_{D_2 } \left\|      \nabla_{1}^{m_{1}} \nabla_2^{k_2} f (\cdot, w_2)\right\|^p_{L^p(D_{ 1})}d\nu_{w_2} 
        \lesssim   \|f\|^p_{W^{k, p}(\Omega)}.
    \end{split}
\end{equation*}
The estimate for $P_1$ is done similarly with an application of Theorem \ref{Pb}.
 

\end{proof}

In particular, the  proposition states that $T_j$ does not lose Sobolev regularity. This estimate of $T_j$  is also the best that one can expect when $n\ge 2$. This is because  $T_j$ only improves the regularity in the $z_j$ direction and has no smoothing effect on the rest of the variables. 
\medskip

\begin{theorem}\label{Pbn}
Let $\Omega:=D_1\times \cdots  \times D_n \subset \mathbb C^n, n\ge 1$, where each $D_j$ is a  bounded domain in $\mathbb C$ with smooth boundary, $j=1, \ldots, n$.
The Bergman projection $P$ is (or, extends as) a bounded operator in $W^{k,p}(\Omega)$,  $k\in \mathbb Z^+\cup\{0\}$,   $ 1< p<\infty$. Namely,  for any $f\in W^{k, p}(\Omega)$, 
\begin{equation*}
    \|Pf\|_{W^{k, p}(\Omega)}\lesssim\|f\|_{W^{k, p}(\Omega)}.
\end{equation*}
\end{theorem}

\begin{proof}[Proof of Theorem \ref{Tbn} and Theorem \ref{Pbn}:]
The proof to Theorem \ref{Tbn} is a direct consequence of Proposition \ref{TPj} and \eqref{Tn}; the proof to  Theorem \ref{Pbn}  is  a direct consequence of Proposition \ref{TPj} and \eqref{Pn}.

\end{proof}

 Denote by $\triangle^2$ the bidisc in $\mathbb C^2$. The following Kerzman-type example demonstrates that the $\bar\partial$ problem in general does not  improve the Sobolev regularity. 
In this sense the Sobolev estimate of the canonical solution operator in Theorem \ref{Tbn} is sharp.

\begin{example}\label{ex}
  For each  $k\in \mathbb Z^+\cup\{0\}$ and $ 1<p<\infty$, consider $  f= (z_2-1)^{k-\frac{2}{p}}d\bar z_1 $ on ${\triangle^2}$ if $ p\ne 2$, or $  f= (z_2-1)^{k-1}\log (z_2-1)d\bar z_1 $ on ${\triangle^2}$ if $ p = 2$, $\frac{1}{2}\pi <\arg (z_2-1)<\frac{3}{2}\pi$. Then  $  f\in W^{k,  q }({\triangle^2} )$ for all $1< q  <p$, and   is  $\bar\partial$-closed on ${\triangle^2}$. However, there does not exist a solution $u\in W^{k, p}({\triangle^2})$ to $\bar\partial u =  f$ on ${\triangle^2}$.  \end{example}

\begin{proof}One can directly verify that $  f\in W^{k,  q }({\triangle^2}) $ for all $1< q  <p$ and is $\bar\partial$-closed on ${\triangle^2}$.  Suppose there exists some $u\in W^{k, p }({\triangle^2} )$ satisfying $\bar\partial u =  f $ on ${\triangle^2}$. Then  $u = (z_2-1)^{k-\frac{2}{p}}\bar z_1+h \in W^{k, p }({\triangle^2} )$ for  some holomorphic function $h$ on ${\triangle^2} $. 
For each $(r, z_2) \in U: = (0,1) \times \triangle\subset \mathbb R^3$, consider  
    $$v(r, z_2): =\int_{|z_1|= r} {u}(z_1, z_2) dz_1. $$
   By  Fubini theorem and H\"older inequality,   \begin{equation*}
        \begin{split}
            \|\partial_{2}^k v\|^{p }_{L^{p }(U )} =&\int_{U}  \left|\int_{|z_1|= r} \partial_{ 2}^k{u}(z_1, z_2) dz_1\right|^{p } d\nu_{z_2} dr 
            =   \int_{|z_2|<1}\int_{0}^1\left|r  \int_{0}^{2\pi} |\partial_{ 2}^k{u}(re^{i\theta}, z_2 )| d\theta  \right|^{p } dr d\nu_{z_2} \\
            \lesssim & \int_{|z_2|<1}\int_0^1 \int_{0}^{2\pi} |{\partial_{ 2}^k u}(re^{i\theta}, z_2 )|^{p }d\theta r dr   d\nu_{z_2}  \le \|{u}\|^{p }_{W^{k, p }({\triangle^2} )}<\infty.
        \end{split}
    \end{equation*} 
Thus  $\partial_{ 2}^k v\in L^{ p }(U)$.

On the other hand, by Cauchy's theorem, for each $(r, z_2)\in U$,
  \begin{equation*}
    \begin{split}
    &\partial_{ 2}^k v(r, z_2) =C_{k,p}\int_{|z_1|=r}  (z_2-1)^{-\frac{2}{p}}\bar z_1dz_1 =   C_{k,p}(z_2-1)^{-\frac{2}{p}}\int_{|z_1|=r} \frac{r^2}{ z_1}dz_1 = 2\pi C_{k,p}  r^2i  (z_2-1)^{-\frac{2}{p}}  
   \end{split}
  \end{equation*}
  for some non-zero constant $C_{k,p}$ depending only on $k$ and $p$. However,  $r^2(z_2-1)^{-\frac{2}{p}}\notin L^{ p }(U )$. This is a  contradiction!  
  
\end{proof}



\bibliographystyle{alphaspecial}

\begin{thebibliography}{HD}
 
 \bibitem{BL}{\sc  E. Barletta and M. Landucci:} \emph{Optimal $L^\infty$ estimates for the canonical solution of the CR-equation in the nonsmooth case.} Complex Variables Theory Appl. {\bf 16} (1991), no. 2-3, 93–106. 
 
 \bibitem{CS}{\sc D. Chakrabarti and M.-C. Shaw:}  \emph{The Cauchy-Riemann equations on product domains}, Math. Ann. {\bf 349} (2011), 977--998.
 
 
    \bibitem{CM}{\sc L. Chen and J.  McNeal:} {\em Product domains, multi-Cauchy transforms, and the $\bar\partial$ equation}.   Adv. Math. \textbf{360} (2020), 106930, 42 pp.

\bibitem{DLT} {\sc X. Dong, S.-Y. Li and J. N. Treuer:}\emph{ Sharp pointwise and uniform estimates for $\bar\partial$.} Anal. PDE, To appear.
 
\bibitem{DPZ}{\sc  X. Dong, Y. Pan and Y. Zhang:} \emph{Uniform estimates for the canonical solution to the $\bar{\partial}$-equation on product domains.} Preprint.    {arXiv:2006.14484}.     

 \bibitem{Ev}{\sc L. Evans:} \emph{Partial differential equations.} Second edition. Graduate Studies in Mathematics, 19. American Mathematical Society, Providence, RI, 2010. xxii+749 pp.


\bibitem{JK}{\sc D. Jerison and C.E. Kenig: }  \emph{The inhomogeneous Dirichlet problem in Lipschitz domains.} J. Funct. Anal. {\bf 130} (1995), no. 1, 161–219.

\bibitem{JY}{\sc  M. Jin and Y. Yuan:}  {\em On the canonical solution of $\bar\partial$ on polydiscs}, C. R. Math. Acad. Sci. Paris. {\bf 358}(2020), no. 5, 523--528.

\bibitem {La} {\sc M. Landucci:} \emph{On the projection of $L^2(D)$ into $H(D)$}, Duke Math. J. {\bf 42} (1975), 231--237.

\bibitem{Li}{\sc S.-Y. Li: } \emph{Solving the Kerzman's problem on the sup-norm estimate for $\bar\partial$ on product domains.} Preprint. ArXiv:2211.01507.
  
  
\bibitem{LS}{\sc L. Lanzani and E.  Stein:} \emph{ Szeg\"o and Bergman projections on non-smooth planar domains.} J. Geom. Anal. {\bf 14} (2004), no. 1, 63–86.


\bibitem{Pr}{\sc M. Prats}: {\em
Sobolev regularity of the Beurling transform on planar domains.}
Publ. Mat. {\bf 61} (2017), no. 2, 291--336.


\bibitem{PZ}{\sc Y. Pan and Y. Zhang:}   {\em Weighted Sobolev estimates of the truncated Beurling operator. } Preprint.  arXiv:2210.02613. 

 \bibitem{PZ2}{\sc Y. Pan and Y. Zhang:} {\em Optimal Sobolev  regularity of $\bar\partial$ on the Hartogs triangle. } Preprint. ArXiv:2210.02619. 
 
 \bibitem{Be}{\sc S. R. Bell}:  \emph{The Cauchy transform, potential theory and conformal mapping}. 2nd edition. Chapman \& Hall/CRC, Boca Raton, FL, 2016.

\bibitem{Yu}{\sc Y. Yuan}: {\em   Uniform estimates of the Cauchy-Riemann equations on product domains.} Preprint. arXiv:2207.02592.

\bibitem{YZ}{\sc Y. Yuan and Y. Zhang}: {\em Weighted Sobolev estimates of $\bar\partial$ on  domains covered by polydiscs.} Preprint.  

 \bibitem{Zh}{\sc Y. Zhang:}  {\em Optimal $L^p$  regularity for $\bar\partial$ on the Hartogs triangle.} Preprint.  arXiv:2207.04944.


\end{thebibliography}
 
\fontsize{10}{9}\selectfont

\fontsize{11}{11}\selectfont

\vspace{0.7cm}

\noindent zhangyu@pfw.edu,

\vspace{0.2 cm}

\noindent Department of Mathematical Sciences, Purdue University Fort Wayne, Fort Wayne, IN 46805-1499, USA.\\
\end{document}